\documentclass[10pt]{amsart}
\usepackage{geometry}                
\usepackage{graphicx}
\usepackage{dsfont}
\usepackage{amssymb}
\usepackage{amsmath}
\usepackage{epstopdf}
\usepackage{fullpage}
\usepackage{enumerate}
\usepackage{color,subfig}
\usepackage{amsthm}
\usepackage{placeins}
\usepackage{array}
\usepackage{blkarray}
\usepackage{multirow}
\usepackage{float}
\usepackage{morefloats}
\usepackage{pgfplots}
\input{macros.sty}

\usepackage{placeins}

\usepackage{amsopn}
\let\oldtocsection=\tocsection

\let\oldtocsubsection=\tocsubsection

\let\oldtocsubsubsection=\tocsubsubsection

\renewcommand{\tocsection}[2]{\hspace{0em}\oldtocsection{#1}{#2}\textbf}
\renewcommand{\tocsubsection}[2]{\hspace{1em}\oldtocsubsection{#1}{#2}}
\renewcommand{\tocsubsubsection}[2]{\hspace{2em}\oldtocsubsubsection{#1}{#2}}

\newenvironment{resume}
  {\begin{abstract}}
  {\end{abstract}}

\newcommand{\dv}{\text{\rm div}}
\newcommand{\incle}{\omega_{\sigma,\varepsilon}}
\newcommand{\tr}{\text{\rm tr}}
\newcommand{\e}{\varepsilon}
\newcommand{\I}{\text{\rm I}}

\newcommand{\R}{{\mathbb R}}

\begin{document}
\newtheorem{theorem}{Theorem}[section]
\newtheorem{remark}{Remark}[section]
\newtheorem{definition}{Definition}[section]
\newtheorem{lemma}{Lemma}[section]
\newtheorem{corollary}{Corollary}[section]
\newtheorem{proposition}{Proposition}[section]
\numberwithin{equation}{section}

\title{A connection between topological ligaments in shape optimization and thin tubular inhomogeneities}
\author{
C. Dapogny\textsuperscript{1}
}
\maketitle
\begin{center}
\emph{\textsuperscript{1} Univ. Grenoble Alpes, CNRS, Grenoble INP\footnote{Institute of Engineering Univ. Grenoble Alpes}, LJK, 38000 Grenoble, France},
\end{center}

\rule{\textwidth}{0.4pt}
\begin{abstract}
\noindent In this note, we propose a formal framework
accounting for the sensitivity of a function of the domain with respect 
to the addition of a thin ligament.
To set ideas, we consider the model setting of elastic structures, and we approximate this question by a thin tubular inhomogeneity problem:
we look for the sensitivity of the solution to a partial differential equation posed inside a background medium, and that of a related quantity of interest,
with respect to the inclusion of a thin tube filled with a different material. 
A practical formula for this sensitivity is derived, which lends itself to numerical implementation.
Two applications of this idea in structural optimization are presented.
\end{abstract}
\begin{resume}
\noindent Dans cette note, on introduit une approche formelle visant \`a \'evaluer la sensibilit\'e 
d'une fonction du domaine par rapport \`a la greffe d'un ligament tr\`es fin sur celui-ci. 
Dans le contexte mod\`ele des structures \'elastiques, nous approchons cette question par un probl\`eme de petite inclusion tubulaire : 
on \'etudie la sensibilit\'e de la solution d'une \'equation aux d\'eriv\'ees partielles pos\'ee dans un milieu ambiant, ainsi que celle d'une quantit\'e d'int\'er\^et associ\'ee,
par rapport \`a l'inclusion d\' un tube fin contenant un mat\'eriau distinct de celui du milieu ambiant. 
On obtient une formule explicite pour cette sensibilit\'e, qui se pr\^ete \`a l'impl\'ementation num\'erique. 
Cette id\'ee est illustr\'ee par deux applications en optimisation structurale. 
\end{resume}
\rule{\textwidth}{0.4pt}

\bigskip
\bigskip
\hrule
\tableofcontents
\vspace{-0.5cm}
\hrule
\bigskip
\bigskip

\section{Introduction}\label{sec.intro}

\noindent Most optimal design frameworks rely on a measure of the sensitivity of the objective (and constraint) function 
with respect to `small modifications' of shapes.
One popular method in this direction is that of Hadamard, whereby variations of a shape are understood as perturbations of their boundaries; see e.g. \cite{allaire2004structural,henrot2018shape,pironneau1982optimal,sokolowski1992introduction}.
This information is sometimes combined with topological derivatives, as in \cite{allaire2005structural};
these indicate where internal holes can be beneficially nucleated. 
Conversely, mechanisms to add material to a shape have seldom been investigated. 
In principle, asymptotic expansions similar to those underlying topological derivatives would make it possible to account for the addition of small bubbles of material.
Such floating islands, disconnected from the main structure, are however unefficient and undesirable from the mechanical viewpoint,
and it would actually be more relevant to add \textit{bars}, connecting distant regions of the shape. 

The sensitivity of the solution to a partial differential equation and that of a shape functional 
with respect to the graft of a thin ligament to the considered domain have been studied in 
\cite{nazarov2004topological,nazarov2005topological,nazarov2005self}, under the name of `exterior topological derivative'. 
Unfortunately, the rigorous asymptotic analyses conducted in these works are intricate and difficult to exploit in practice, as the authors themselves acknowledge.

In this note, we propose an alternative, formal approach, which is easier to handle in theory and more amenable to numerical implementation. 
In the model setting of a 2d structure $\Omega$, we use the ersatz material approximation to replace the linear elasticity system on $\Omega$ with a `background' problem 
taking place on a larger hold-all domain $D$, filled with a smooth inhomogeneous material $A_0$.
The addition of a thin ligament to $\Omega$
is then reformulated as the inclusion of a thin tubular inhomogeneity with different material properties $A_1$ from $A_0$. 
The sensitivity of the elastic displacement of the structure and that of a related quantity of interest---the pivotal ingredients of this viewpoint--- can then be calculated by 
borrowing techniques from the literature devoted to low-volume inhomogeneities. 
Such asymptotic problems have indeed been quite extensively investigated; see \cite{capdeboscq2003general}, 
then \cite{beretta2003asymptotic,beretta2009thin} about thin tubular inclusions for the conductivity equation, and \cite{beretta2006asymptotic} in the 2d linearized elasticity case.

This note is preliminary to a longer work \cite{dapogny2020topolig} in preparation, 
where the extension to 3d (the situation being utterly different from that in 2d), as well as multiple other applications 
are discussed, and a general and simple formal method is proposed to calculate the thin tubular inhomogeneity asymptotics.

The remainder of this note is organized as follows. 
The considered setting of linear elastic structures is introduced in \cref{sec.setting}, as 
well as its approximation by a thin tubular inhomogeneity problem. We recall in \cref{sec.ueJe} 
the first-order asymptotic expansion of the solution to the linear elasticity system when the background medium is perturbed by a thin tubular inclusion, 
and we introduce a suitable adjoint method to calculate the first-order correction of a related quantity of interest, which is new to the best of our knowledge.
Two numerical examples illustrating these ideas are eventually presented in \cref{sec.num}.

\section{Presentation of the structural optimization problem and relation with thin tubular inhomogeneities}\label{sec.setting}

\subsection{Optimization of the shape of a 2d elastic structure}\label{sec.elasexact}

\noindent In the 2d linear elasticity setting, shapes 
are bounded, Lipschitz domains $\Omega \subset \R^2$ whose boundary $\partial \Omega= \Gamma_D \cup \Gamma_N \cup \Gamma$ is divided into three disjoint parts:
$\Omega$ is clamped on $\Gamma_D$, traction loads $g \in L^2(\Gamma_N)^2$ are applied on $\Gamma_N$, 
and the traction-free region $\Gamma$ is the only one which is subject to optimization. Assuming body forces $f \in L^2(\R^2)^2$, 
the displacement $u_\Omega: \Omega \to \R^2$ is the unique solution in the space $H^1_{\Gamma_D}(\Omega)^2 := \left\{Êu \in H^1(\Omega)^2, \:\: u = 0 \text{ on } \Gamma_D \right\}$ to the system: 
\begin{equation}\label{eq.linelas}
 \left\{ 
\begin{array}{cl}
-\dv(Ae(u_\Omega)) = f & \text{in } \Omega, \\
u_\Omega = 0 & \text{on } \Gamma_D, \\
Ae(u_\Omega)n = g &\text{on } \Gamma_N, \\
Ae(u_\Omega) = 0 & \text{on } \Gamma,
\end{array}
\right.
\end{equation}
where $e(u) := \frac{1}{2}(\nabla u + \nabla u^T)$ is the strain tensor, 
and $A$ is the Hooke's law of the constituent material: 
\begin{equation}\label{eq.hooke}
\text{for any symmetric } 2 \times 2 \text{ matrix } e,\:\: Ae = 2\mu e + \lambda \tr(e)  \I,
 \end{equation}
involving the Lam\'e coefficients $\lambda,\mu$ of the material. 
The performance of $\Omega$ is measured in terms of a function $J(\Omega)$ of the domain, say for simplicity: 
\begin{equation}\label{eq.JOm}
 J(\Omega) = \int_\Omega{j(u_\Omega) \: dx},
 \end{equation}
where $j: \R^2 \to \R$ is a smooth function satisfying adequate growth conditions. 

We consider the variation $\Omega_{\sigma,\e}$ where a ligament $\omega_{\sigma,\e}$ with thickness $\e \ll 1$ is grafted to $\Omega$: 
$$ \Omega_{\sigma,\e}= \Omega \cup \omega_{\sigma, \e}, \text{ where } \omega_{\sigma,\e} := \left\{Êx \in \R^2, \:\: \text{\rm dist}(x,\sigma) < \e \right\},$$
and $\sigma$ is a smooth, non self-intersecting curve in $\R^2$ whose endpoints belong to $\partial \Omega$; see \cref{fig.incl2set} (left). 
Assuming homogeneous Neumann boundary conditions on $\partial \omega_{\sigma,\e}$ in the defining system for $u_{\Omega_{\sigma,\e}}$ (the version of \cref{eq.linelas} posed on $\Omega_{\sigma,\e}$), we look for an expansion of the form:
\begin{equation}\label{eq.ligasym}
 J(\Omega_{\sigma,\e}) = J(\Omega) + \e dJ_L(\Omega)(\sigma) + o(\e),
 \end{equation}
where it is tempting to call the first-order term $dJ_L(\Omega)(\sigma)$ the `ligament derivative' of $J(\Omega)$.

\subsection{The thin tubular inhomogeneity problem}

\noindent We replace \cref{eq.linelas} with the following equation, taking place in a fixed `hold-all' domain $D$:
\begin{equation}\label{eq.elasbg}
 \left\{ 
\begin{array}{cl}
-\dv(A_0e(u_0)) = f & \text{in } D, \\
u_0 = 0 & \text{on } \Gamma_D, \\
A_0e(u_0)n = g &\text{on } \Gamma_N, \\
A_0e(u_0) = 0 & \text{on } \Gamma,
\end{array}
\right.
\end{equation}
where $A_0$ is a smooth Hooke's tensor of the form \cref{eq.hooke} with inhomogeneous coefficients $\lambda_0(x),\mu_0(x)$. 
This problem is an approximation of that in \cref{sec.elasexact} 
if $A_0$ is defined as a smooth transition between the Hooke's tensor $A$ inside $\Omega$ 
and that of a very soft material $\eta A$, $\eta \ll 1$, inside the void $D\setminus \Omega$ (this is the classical ersatz material method; see \cite{allaire2004structural}); see \cref{fig.incl2set} (right).

The perturbed version of \cref{eq.elasbg} where a tube $\omega_{\sigma,\e}$ filled by another material $A_1$ with Lam\'e parameters $\lambda_1(x)$, $\mu_1(x)$ is included in $D$ is:
\begin{equation}\label{eq.inhpert}
 \left\{ 
\begin{array}{cl}
-\dv(A_\e e(u_\e)) = f & \text{in } D, \\
u_\e = 0 & \text{on } \Gamma_D, \\
A_\e e(u_\e)n = g &\text{on } \Gamma_N, \\
A_\e e(u_\e) = 0 & \text{on } \Gamma,
\end{array}
\right. 
\text{ with } A_\e (x) = \left\{ 
\begin{array}{cl}
A_1(x) & \text{if } x \in \incle, \\
A_0(x) & \text{otherwise}.
\end{array}
\right.
\end{equation}
The approximate counterparts $J_\sigma(0)$ and $J_\sigma(\e)$ of the functionals $J(\Omega)$ and $J(\Omega_{\sigma,\e})$ in \cref{eq.JOm} read:  
\begin{equation}\label{eq.Jsigmae}
J_\sigma(0) = \int_D{j(u_0)\:dx}, \:\: J_\sigma(\e) = \int_D{j(u_\e) \: dx},
\end{equation}
and we approximate $dJ_L(\Omega)(\sigma)$ in \cref{eq.ligasym} by the first-order term $J_\sigma^\prime(0)$ in the expansion:
$$ J_\sigma(\e) =J_\sigma(0) + \e J_\sigma^\prime(0) + o(\e). $$

\begin{figure}[!ht]
\centering
\includegraphics[width=1.0\textwidth]{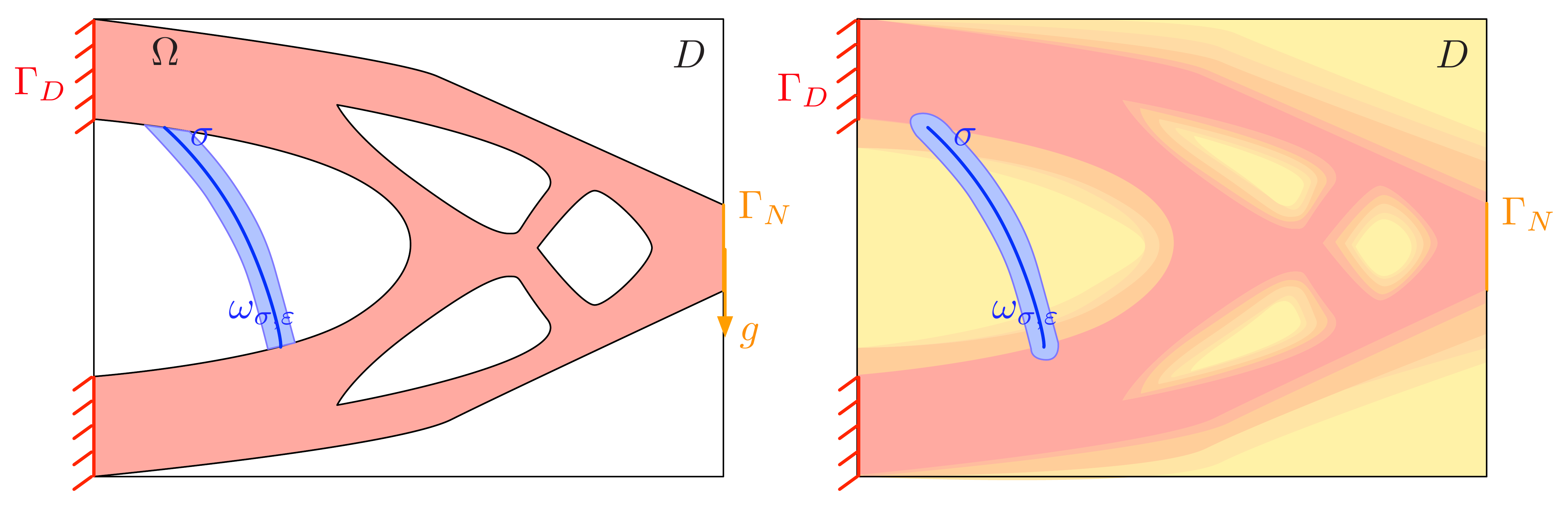}
\caption{\it (Left) Graft of a ligament with base curve $\sigma$ to an elastic structure $\Omega$; 
(right) corresponding thin tubular inclusion inside an approximating background medium occupying the larger domain $D$.}
\label{fig.incl2set}
\end{figure}

\section{An adjoint method for the topological ligament}\label{sec.ueJe}

\noindent In this section, we discuss asymptotic formulas for $u_\e$ and $J_\sigma(\e)$ in \cref{eq.inhpert,eq.Jsigmae} as $\e \to 0$. 

\subsection{Asymptotic formula for the state $u_\e$}

\noindent
Let the $2\times 2$ matrix field $N(x,y) = \left\{ N_{ij}(x,y)\right\}_{i=1,2 \atop j=1,2}$ be the fundamental solution of the system \cref{eq.elasbg}.
More precisely, for any $x \in D$, the $j^{\text{\rm th}}$ column vector $y \mapsto N_j(x,y) = \left\{N_{ij}(x,y) \right\}_{i=1,2}$ satisfies: 
$$\left\{ 
\begin{array}{cl}
 \dv_y(A_0e_y(N_j(x,y))) = \delta_{y=x} e_j & \text{in } D, \\
 A_0e_y(N_j(x,y))n(y) = 0 & \text{on }Ê\Gamma_N \cup \Gamma, \\
 N_j(x,y) = 0 & \text{on } \Gamma_D.
 \end{array}
 \right.
 $$
where $e_j$ is the $j^{\text{\rm th}}$ vector in the canonical basis of $\R^2$. The main result of interest is the following:

\begin{theorem}\label{th.asymue}
Let $x \in \overline{D} \setminus \sigma$; the solution $u_\e$ to the perturbed system \cref{eq.inhpert} fulfills the following expansion:
\begin{equation}\label{eq.asymue}
 \frac{1}{\e}(u_\e - u_0)(x) = u_1 + o(1), \text{ where } u_1(x) = \int_{\sigma}{{\mathcal M}(y) e(u_0)(y) : e_y(N(x,y)) \:d \ell(y)},
 \end{equation}
and the remainder $o(1)$ is uniform when $x$ is confined to a fixed compact subset $K \subset  \overline{D} \setminus \sigma$. 
The polarization tensor ${\mathcal M}(y)$ reads, for any symmetric $2\times 2$ matrix $e$: 
\begin{equation}\label{eq.M}
 {\mathcal M}(y) e = \alpha(y) \tr(e) \I + \beta(y) e + \gamma(y) (e \tau \cdot \tau ) \tau \otimes \tau + \rho(y) (e n \cdot n) n \otimes n,
 \end{equation}
involving the (inhomogeneous) coefficients: 
$$ \alpha = 2(\lambda_1-\lambda_0) \frac{\lambda_0 + 2\mu_0}{\lambda_1 + 2\mu_1}, \:\: \beta = 4(\mu_1- \mu_0) \frac{\mu_0}{\mu_1}, $$
and 
$$\gamma = 4(\mu_1-\mu_0) \left( \frac{2\lambda_1 + 2\mu_1 - \lambda_0}{\lambda_1 + 2\mu_1} -\frac{\mu_0}{\mu_1}\right), \:\: \rho = 4(\mu_1 - \mu_0) \frac{\mu_1\lambda_0- \mu_0 \lambda_1}{\mu(\lambda_1 + 2\mu_1)} .$$
\end{theorem}

The rigorous (difficult) proof of this result is given in \cite{beretta2006asymptotic}. 
Interestingly, this formula can also be obtained thanks to a formal method based on energy minimization, 
close to that used in \cite{nguyen2009representation,dapogny2017uniform}.

\subsection{An adjoint state method for the derivative of an observable}

\noindent We are now in position to derive the behavior of the functional $J_\sigma(\e)$ in \cref{eq.Jsigmae} as $\e\to 0$.

\begin{theorem}\label{th.asymje}
The following expansion holds:
\begin{equation}\label{eq.Jsigmap}
 J_\sigma(\e ) = J_\sigma(0) + \e J_\sigma^\prime(0) + o(1), \text{ where }ÊJ^\prime_\sigma(0) := \int_\sigma{{\mathcal M}(y) e(u_0):e(p_0) \:d \ell(y)},
\end{equation}
where ${\mathcal M}(y)$ is the polarization tensor in \cref{eq.M}, and the adjoint state $p_0 \in H^1_{\Gamma_D}(\Omega)^2$ is the solution to:
\begin{equation}\label{eq.adj} 
\left\{Ê
\begin{array}{cl}
-\dv (A_0 e(p_0)) = -j^\prime(u_0) & \text{in } D, \\
p_0 = 0 & \text{on } \Gamma_D, \\
Ae(p_0) n = 0 & \text{on } \Gamma_N \cup \Gamma.
\end{array}
\right.
\end{equation}
\end{theorem}
\begin{proof}[Sketch of proof:]
A variant of the Aubin-Nitsche trick (see e.g. \cite{ciarlet2002finite}) allows to show that the limit $J_\sigma^\prime(0) = \lim_{\e \to 0}{\frac{J_\sigma(\e) - J_\sigma(0)}{\e}}$ exists
and has the expression: 
$$ J_\sigma^\prime(0) = \int_D{j^\prime(u_0) u_1 \:dx},$$
where $u_1$ is the first-order term in \cref{eq.asymue}. 
Now introducing the adjoint state \cref{eq.adj}, we obtain: 
$$ \begin{array}{>{\displaystyle}cc>{\displaystyle}l}
J_\sigma^\prime(0) &=& \int_D{j^\prime(u_0(x)) \left(\int_\sigma{{\mathcal M}(y) e(u_0)(y) : e_y(N(x,y)) \:d \ell(y)} \right)\: d x}.\\
&=& \int_\sigma{{\mathcal M}(y) e(u_0)(y):e_y\left( \int_D{j^\prime(u_0(x)) N(x,y) \:d x}\right) \: d \ell(y)} \\
&=& \int_\sigma{{\mathcal M}(y) e(u_0)(y):e(p_0)(y) \:d \ell(y)}, 
\end{array}
$$
where we have used the integral representation formula: 
$$ p_0(y) = \int_D{j^\prime(u_0(x))N(x,y) \: dx}.$$
\end{proof}

\subsection{Practical interest of the result}\label{sec.practinterest}

\noindent We return to our purpose of finding a curve $\sigma$ such that the variation $\Omega_{\sigma,\e}$ of a given shape $\Omega$ 
achieves a lower value $J(\Omega_{\sigma,\e}) < J(\Omega)$. 
According to the discussion of \cref{sec.setting}, 
we consider the background medium $A_0$ in $D$ obtained from $\Omega$ via the ersatz material approximation, 
and we search for $\sigma$ such that $J_\sigma^\prime(0)<0$. 

Using \cref{th.asymje}, and assuming for simplicity that $\sigma$ is a line segment with (constant) tangent vector $\tau = (\tau_1,\tau_2) \in \R^2$, 
formula \cref{eq.Jsigmap} can be rewritten:
$$ J_\sigma^\prime(0) = \int_\sigma{P(y,\tau_1,\tau_2) \:d\ell (y)},$$
where for a given point $y \in \sigma$, $(\tau_1,\tau_2) \mapsto P(y,\tau_1,\tau_2)$ is a bivariate $4^{\text{\rm th}}$-order homogeneous polynomial, 
whose coefficients depend explicitly on $y$ via the entries of $e(u_0)(y)$, $e(p_0)(y)$. 

The search for an `optimal' line segment $\sigma$ such that $J_\sigma^\prime(0)<0$ is then achieved along the following lines: 
\begin{enumerate}
\item Solve \cref{eq.elasbg} and \cref{eq.adj} (e.g. by the finite element method) for $u_0$ and $p_0$, respectively. 
\item Calculate the coefficients of the polynomial $P(y,\cdot,\cdot)$ using the formulas in \cref{th.asymue,th.asymje}; 
\item For all points $z_1, z_2$ in (a discretization of) $\partial \Omega$, calculate $J_{\sigma}^\prime(0)$, when $\sigma$ is the line segment $[z_1,z_2]$, and
retain the couple $(z_1,z_2)$ achieving the negative value of $J_{\sigma}^\prime(0)$ with largest modulus.
\end{enumerate}
Note that the step (3) in this program is relatively unexpensive since the involved quantities---and notably the coefficients of $P(y,\cdot,\cdot)$ for $y \in D$---are computed beforehand, once and for all.

\section{Numerical algorithm}\label{sec.num}

\noindent We finally apply the methodology of \cref{sec.practinterest} to two shape optimization problems of the form: 
\begin{equation}\label{eq.sopb}
 \min\limits_\Omega {J(\Omega)}, \text{ s.t. } C(\Omega) = 0.
 \end{equation}
Both examples are addressed using Hadamard's boundary variation method (see again \cite{allaire2004structural,henrot2018shape,pironneau1982optimal,sokolowski1992introduction}). 
We track the motion of the shape $\Omega$ thanks to the level set-based mesh evolution method from \cite{allaire2014shape};
this allows for an explicit, meshed representation of $\Omega$ at each stage of the process, and no ersatz material approximation is needed
to compute shape gradients.
The constrained optimization in \cref{eq.sopb} is treated by the null-space algorithm from \cite{feppon2019null}. 
We enrich this classical framework with the addition of material ligaments to $\Omega$ using the methodology of \cref{sec.practinterest} in two different ways.

\subsection{Adding bars in the course of the shape optimization process}\label{sec.canti}

\noindent In this section, we seek to add bars to the shape in order to enrich its topology in the course of a `classical' shape optimization process driven by the method of Hadamard. 
The physical setting is that of the cantilever test-case, as depicted on \cref{fig.cantiex} (top, left). 
Shapes are contained in a box with size $2\times 1$; they are attached on the left-hand side of their boundary, 
and a unit vertical load $g = (0,-1)$ is applied on a region $\Gamma_N$ at the middle of their right-hand side. 
Omitting body forces for simplicity, we aim to minimize the compliance of the shape $\Omega$ under a volume constraint, i.e. we solve \cref{eq.sopb} with:
$$J(\Omega) = \int_{\Gamma_N}{g \cdot u_\Omega \:ds},\:\: C(\Omega)= \int_\Omega {\:dx} - V_T, \text{ and the volume target } V_T = 0.8.$$ 

Starting from the shape in \cref{fig.cantiex} (top, left), we consider a situation where the parameters of the optimization algorithm 
are tuned so that the volume constraint tends to be satisfied `too fast':
the holes merge permaturely (\cref{fig.cantiex}, (top,right)), and the resulting shape has a trivial topology, with a large value $4.115$ of the compliance (computation not reported).
We consider the same test-case, except that every 10 iterations, from iteration $40$ to $100$, the procedure in \cref{sec.practinterest} is used to graft bars to the structure. 
The resulting shape has a richer topology, for an improved final value of the compliance $2.718$; see \cref{fig.cantiex} for snapshots of the evolution process and \cref{fig.cantiexcv} for the associated convergence histories.
Note that some of the bars created during the topological ligament steps eventually disappear after some iterations.

\begin{figure}[!ht]
\begin{tabular}{cc}
\begin{minipage}{0.5\textwidth}
\hspace{-1.2cm}\includegraphics[width=1.195\textwidth]{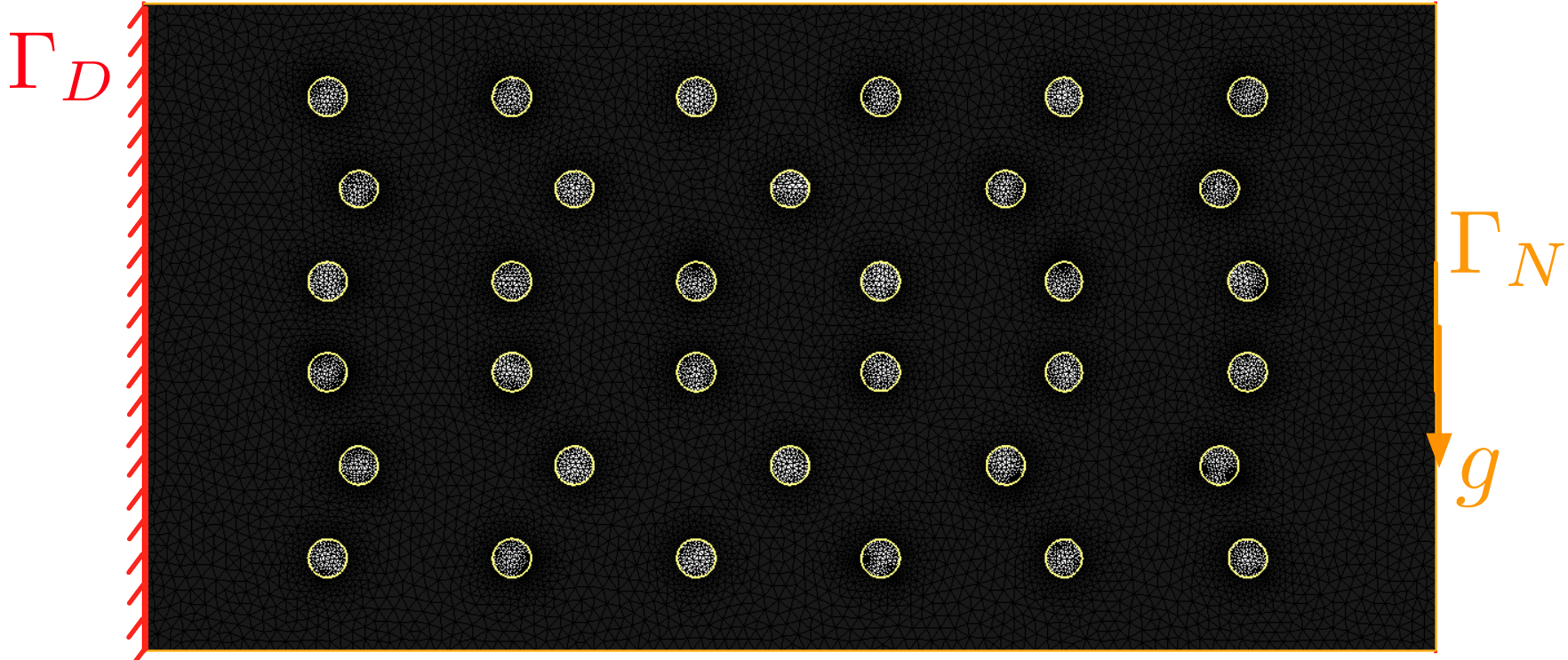}
\end{minipage} & \begin{minipage}{0.5\textwidth}
\includegraphics[width=1.0\textwidth]{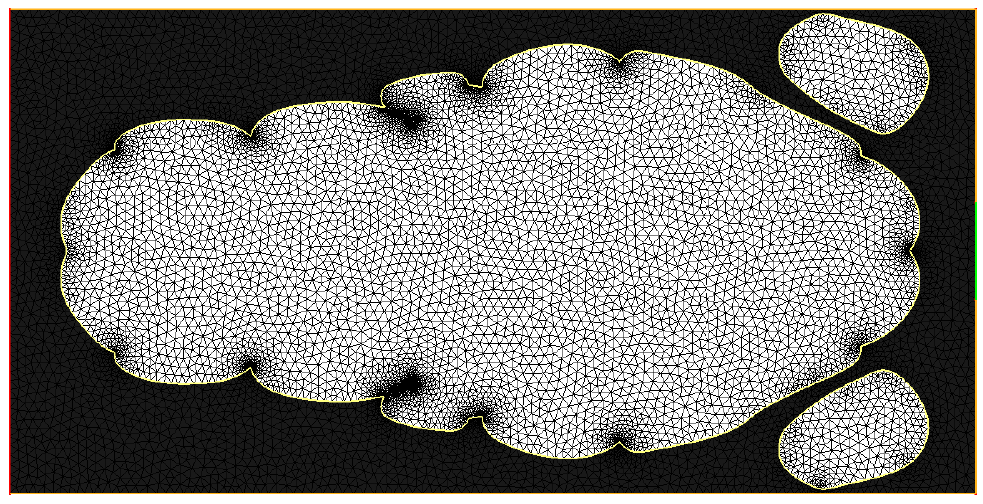}
\end{minipage} \\Ê
\begin{minipage}{0.5\textwidth}
\hspace{-0.4cm}\includegraphics[width=1.0\textwidth]{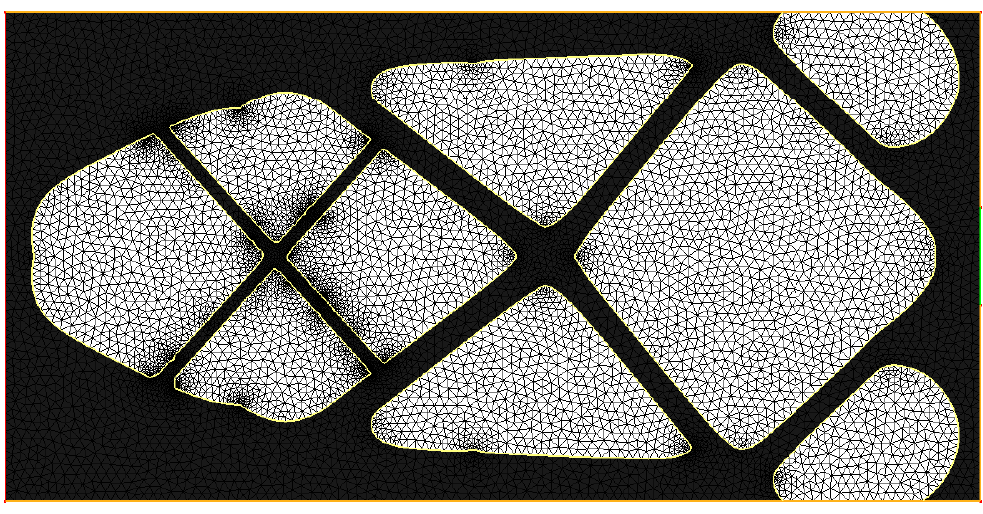}
\end{minipage} & \begin{minipage}{0.5\textwidth}
\includegraphics[width=1.0\textwidth]{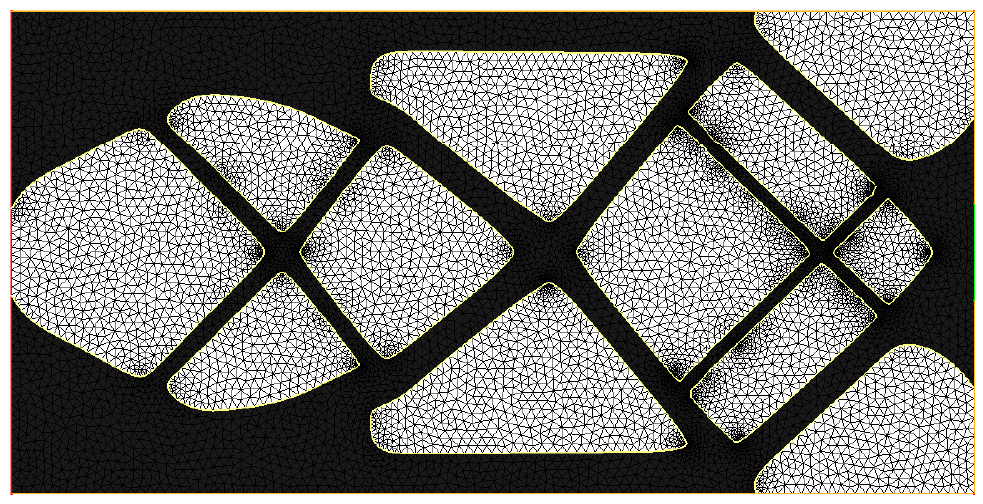}
\end{minipage} \\Ê
\begin{minipage}{0.5\textwidth}
\hspace{-0.4cm}\includegraphics[width=1.0\textwidth]{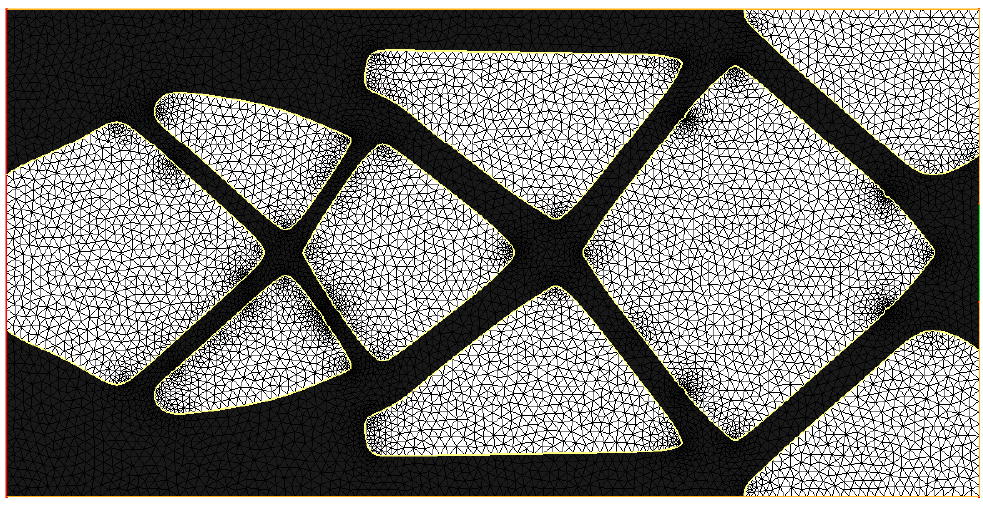}
\end{minipage} & \begin{minipage}{0.5\textwidth}
\includegraphics[width=1.0\textwidth]{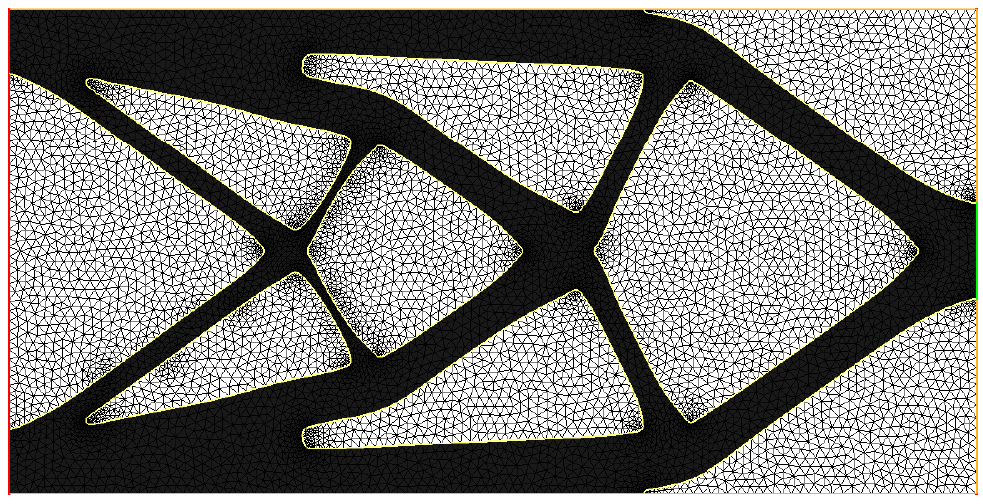}
\end{minipage} \\Ê
\end{tabular}
\caption{\it From left to right, top to bottom: Iterations 0 (with details of the test-case), 40, 51, 61, 85, and 200 of the cantilever test-case of \cref{sec.canti}.}
\label{fig.cantiex}
\end{figure}

\begin{figure}[!ht]
\begin{tabular}{cc}
\begin{minipage}{0.35\textwidth}
\includegraphics[width=1.0\textwidth]{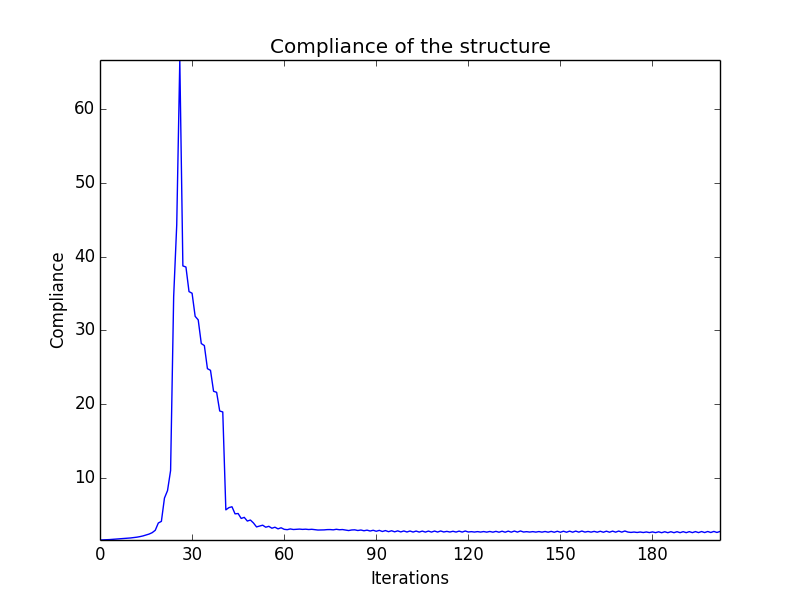}
\end{minipage} & \begin{minipage}{0.35\textwidth}
\includegraphics[width=1.0\textwidth]{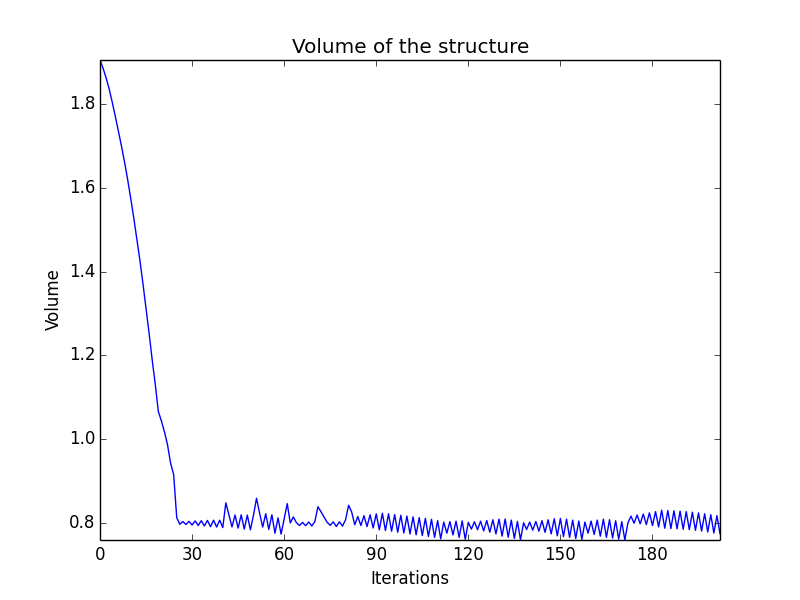}
\end{minipage} 
\end{tabular}
\caption{\it Convergence histories for the compliance (left) and the volume (right) of the shape in the cantilever example of \cref{sec.canti}.}
\label{fig.cantiexcv}
\end{figure}

\subsection{A judicious initialization for truss-like structures}\label{sec.mastex}

\noindent Our strategy can also be used as a preprocessing of 
a traditional shape optimization process, in situations where truss-like structures (i.e. containing lots of bars) are expected. 
In the setting of a T-shaped mast (see \cref{fig.mastex} (top, left)) we consider the minimization \cref{eq.sopb} of the volume under a compliance constraint:
$$ J(\Omega) = \int_\Omega{dx}, \text{ and } C(\Omega) = \int_{\Gamma_N}{g \cdot u_\Omega \:ds} - C_T, \text{ and the target value } C_T = 0.4.$$
Starting from an empty shape, a first stage aims to enrich the structure with bars of the form $[z_1,z_2]$, 
connecting endpoints $z_1,z_2$ sought within a user-defined set of points in $D$ (in red in \cref{fig.mastex} (top, left)), following the methodology of \cref{sec.practinterest}.
The process ends when the compliance of the structure is close to the target $C_T$ (in our case, when it reaches the value $1.25$);
the resulting shape  is that in \cref{fig.mastex} (top, right). 

In a second stage, we use this structure as the initialization for the resolution of \cref{eq.sopb} by means of the Hadamard's boundary variation method; see \cref{fig.mastex} (bottom row) and \cref{fig.mastexcv} for the convergence histories.

\begin{figure}[!ht]
\begin{tabular}{ccc}
\begin{minipage}{0.26\textwidth}
\hspace{-1cm}\includegraphics[width=1.15\textwidth]{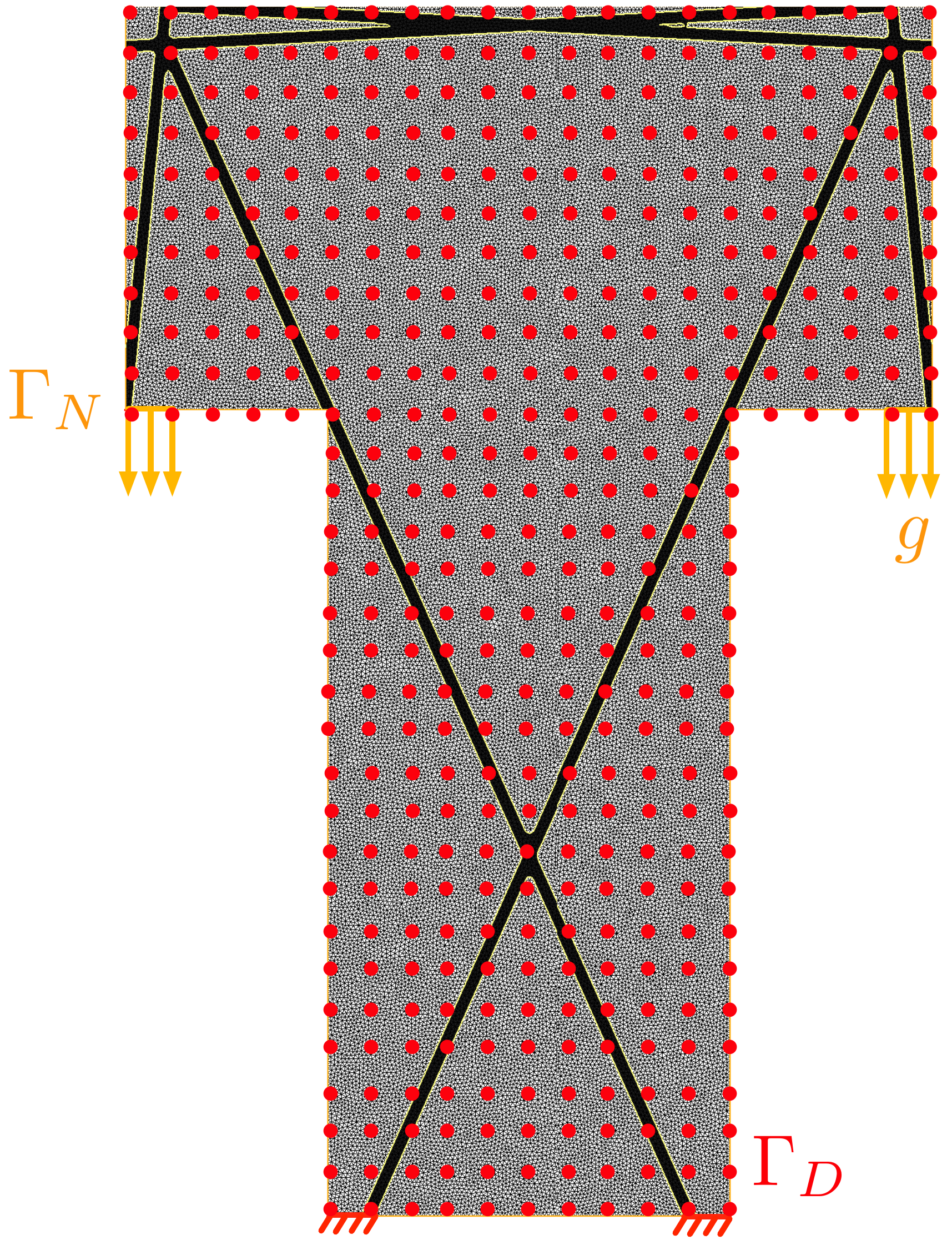}
\end{minipage} &
 \begin{minipage}{0.26\textwidth}
\includegraphics[width=1.0\textwidth]{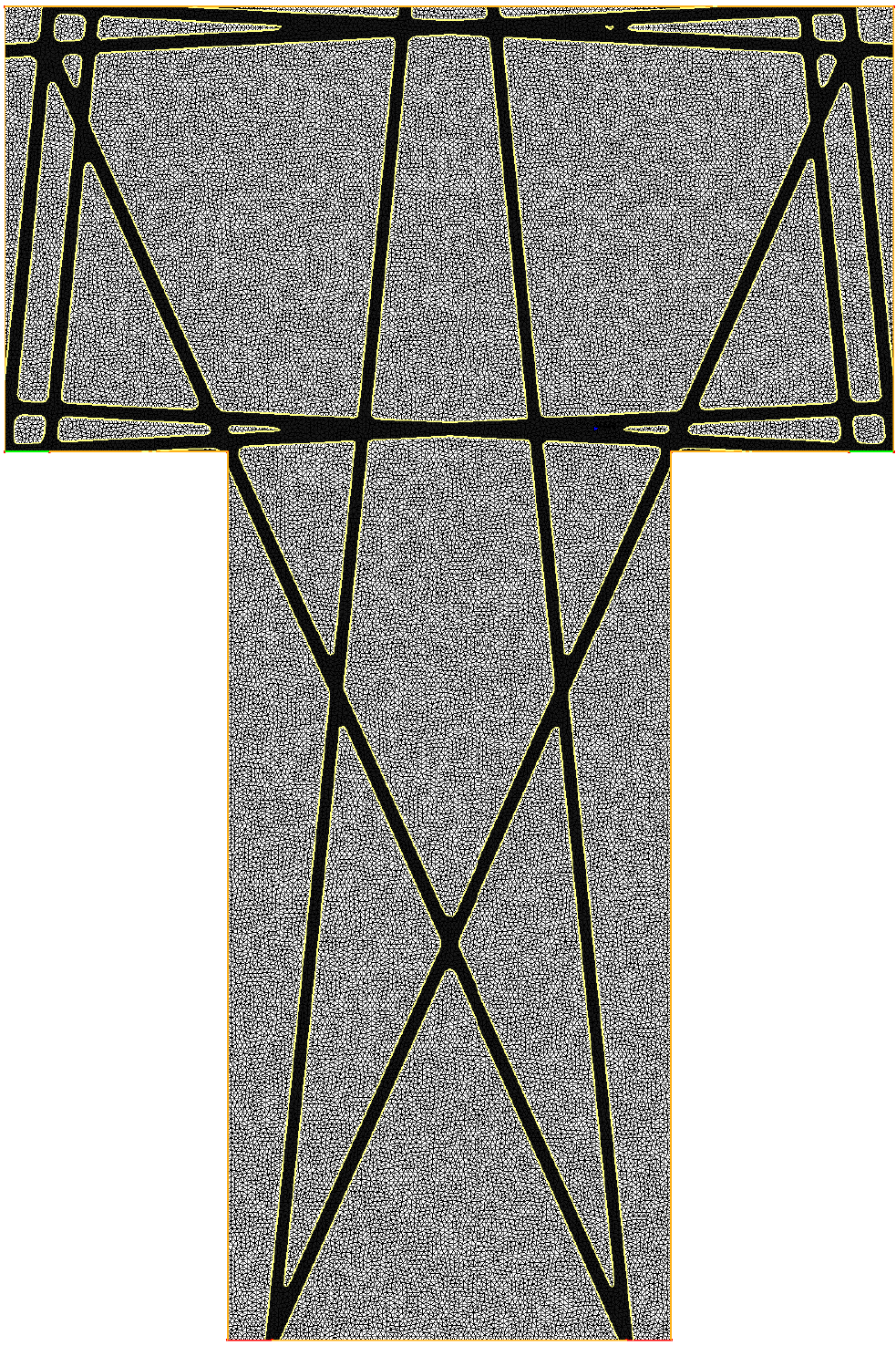}
\end{minipage} &
\begin{minipage}{0.26\textwidth}
\includegraphics[width=1.0\textwidth]{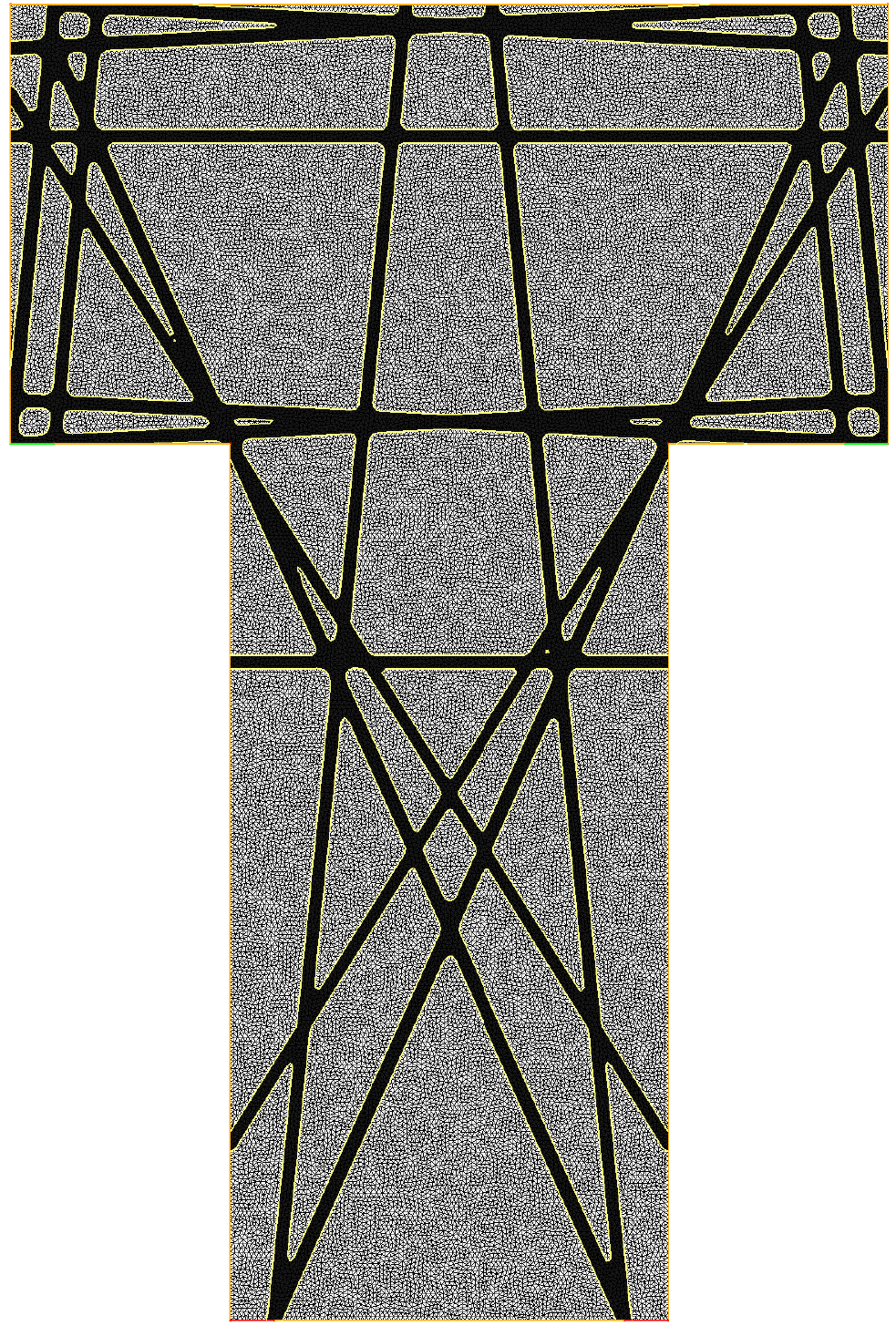}
\end{minipage} \\
\begin{minipage}{0.26\textwidth}
\includegraphics[width=1.0\textwidth]{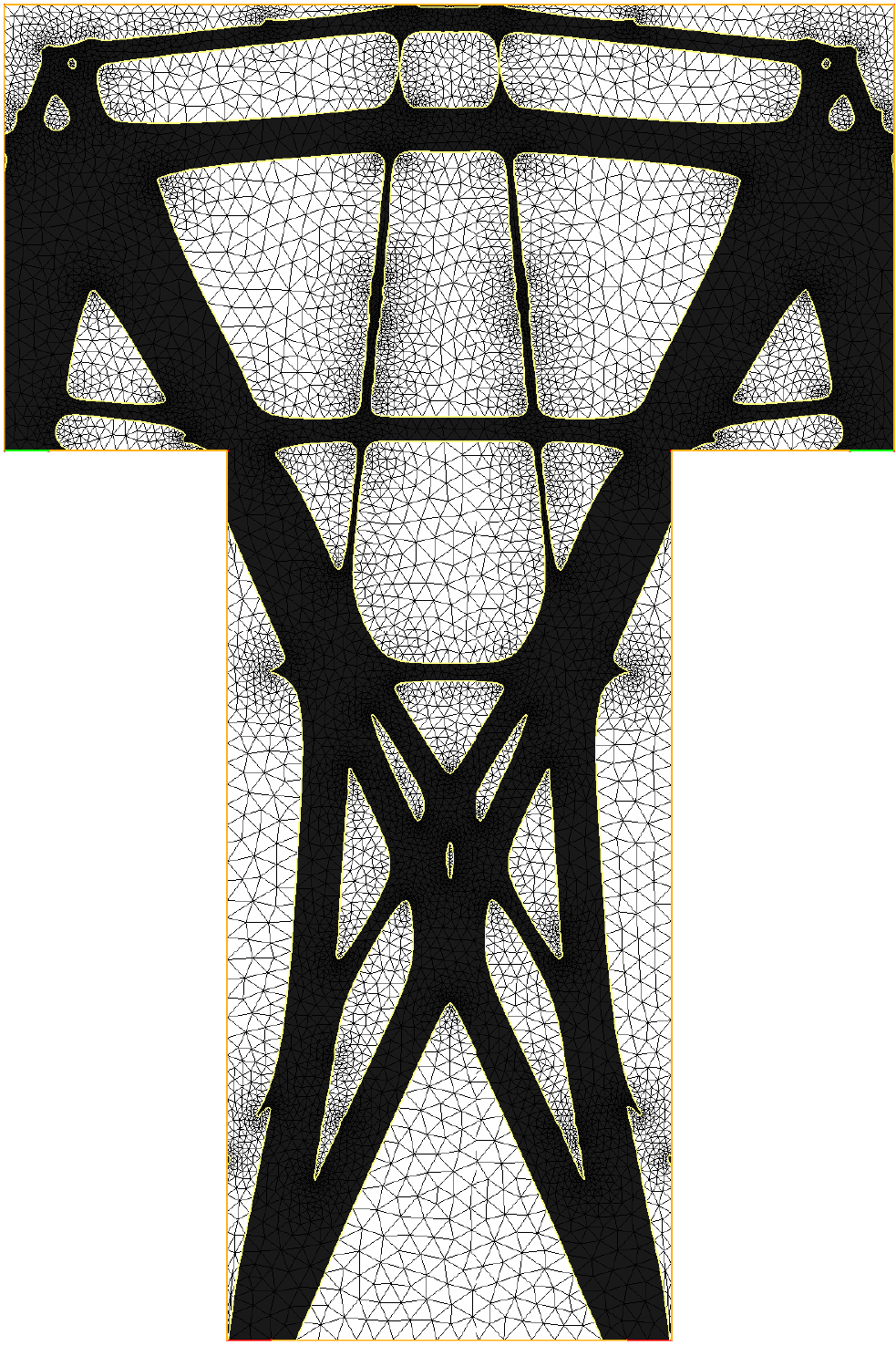}
\end{minipage} &
 \begin{minipage}{0.26\textwidth}
\includegraphics[width=1.0\textwidth]{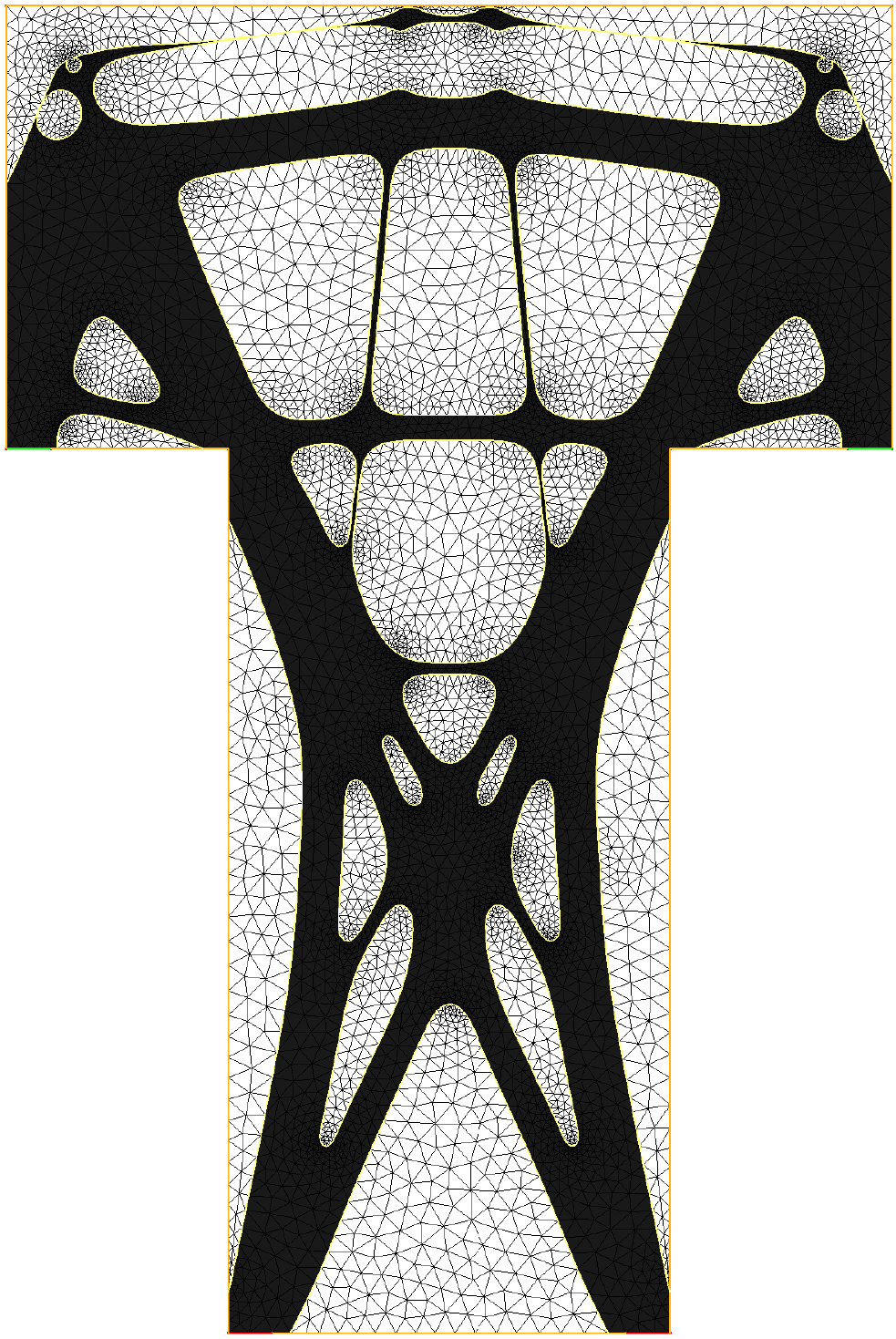}
\end{minipage} &
\begin{minipage}{0.26\textwidth}
\includegraphics[width=1.0\textwidth]{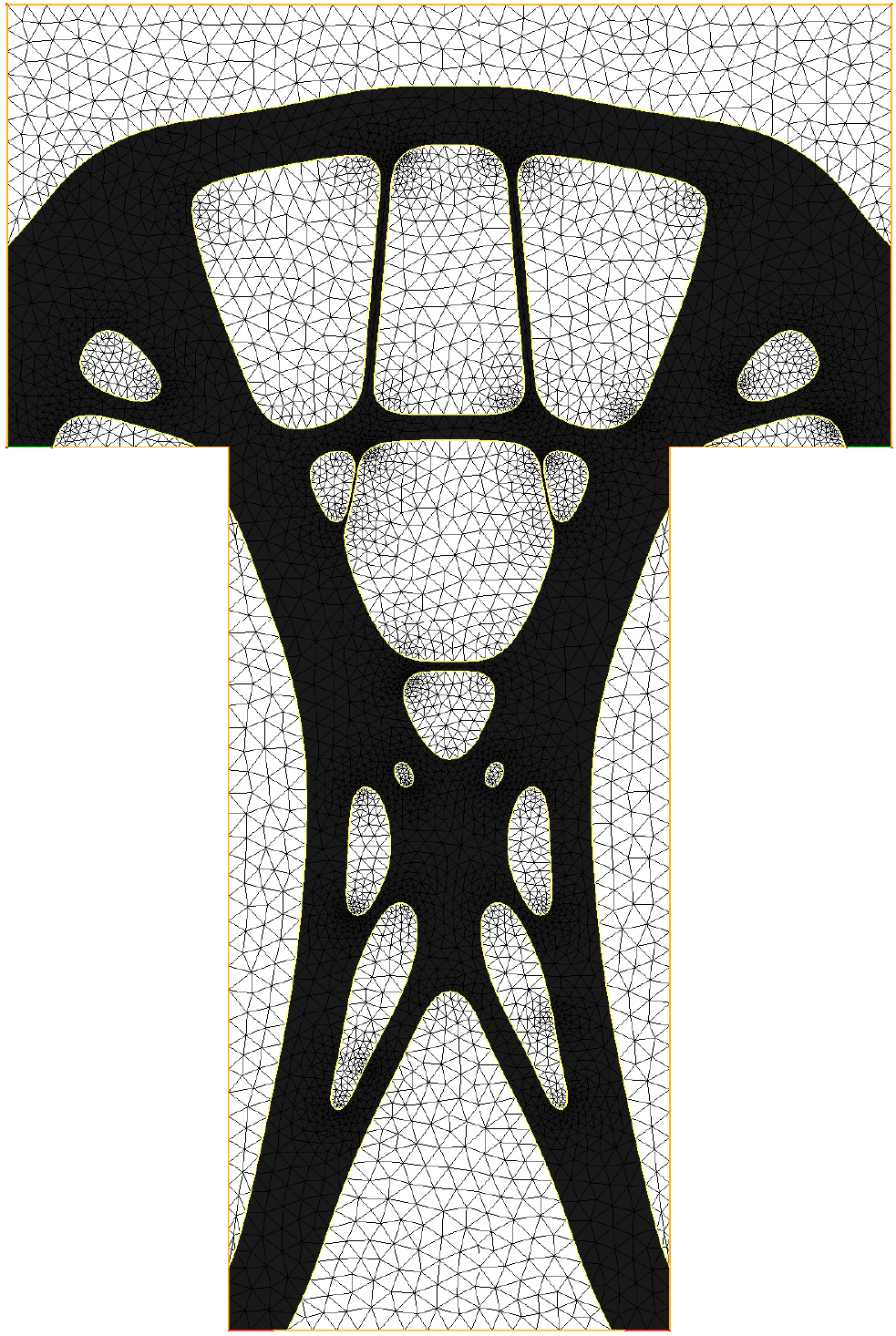}
\end{minipage} \\
\end{tabular}
\caption{\it (Upper row, from left to right): Steps 3 (with details of the test-case), 6, and 9 of the first stage of bar insertion; (lower row, from left to right) steps 10, 60 and 200 of the second stage, in the mast example of \cref{sec.mastex}.}
\label{fig.mastex}
\end{figure}

\begin{figure}[!ht]
\begin{tabular}{ccc}
\begin{minipage}{0.3\textwidth}
\includegraphics[width=1.0\textwidth]{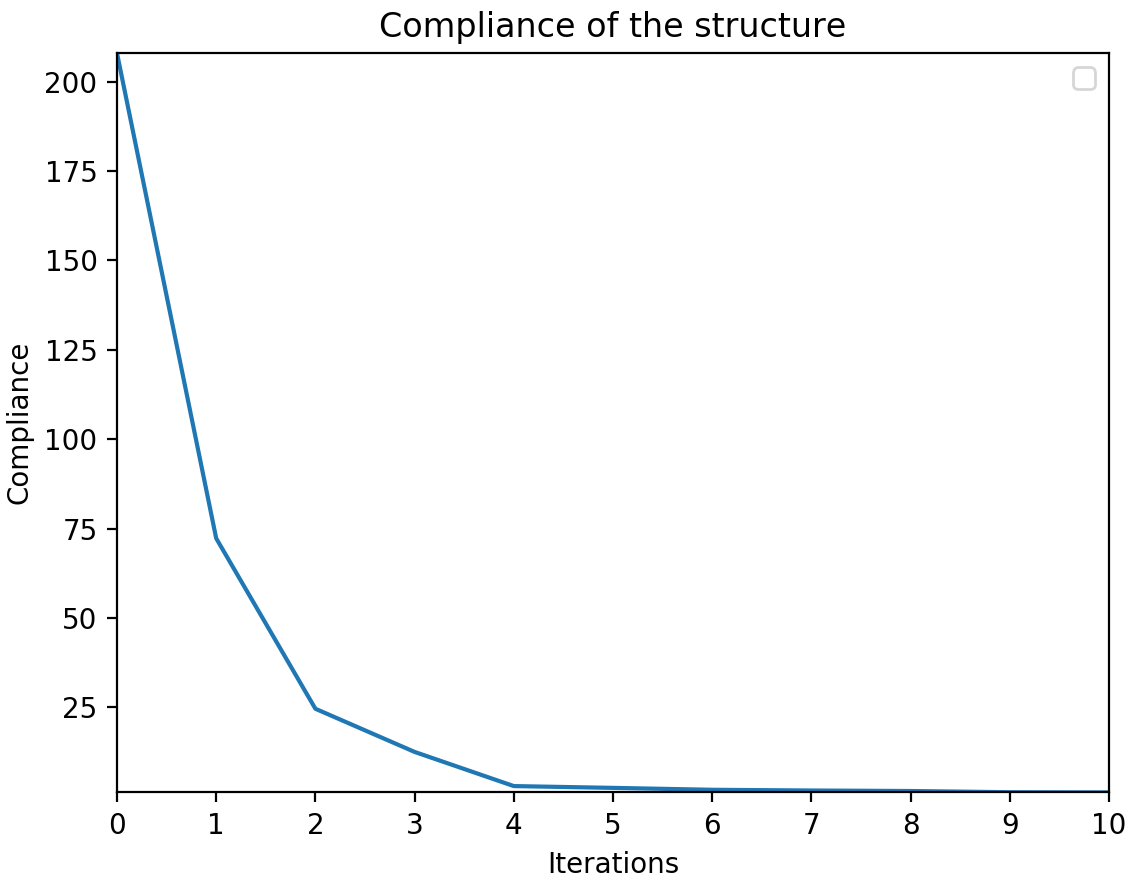}
\end{minipage} & \begin{minipage}{0.3\textwidth}
\includegraphics[width=1.0\textwidth]{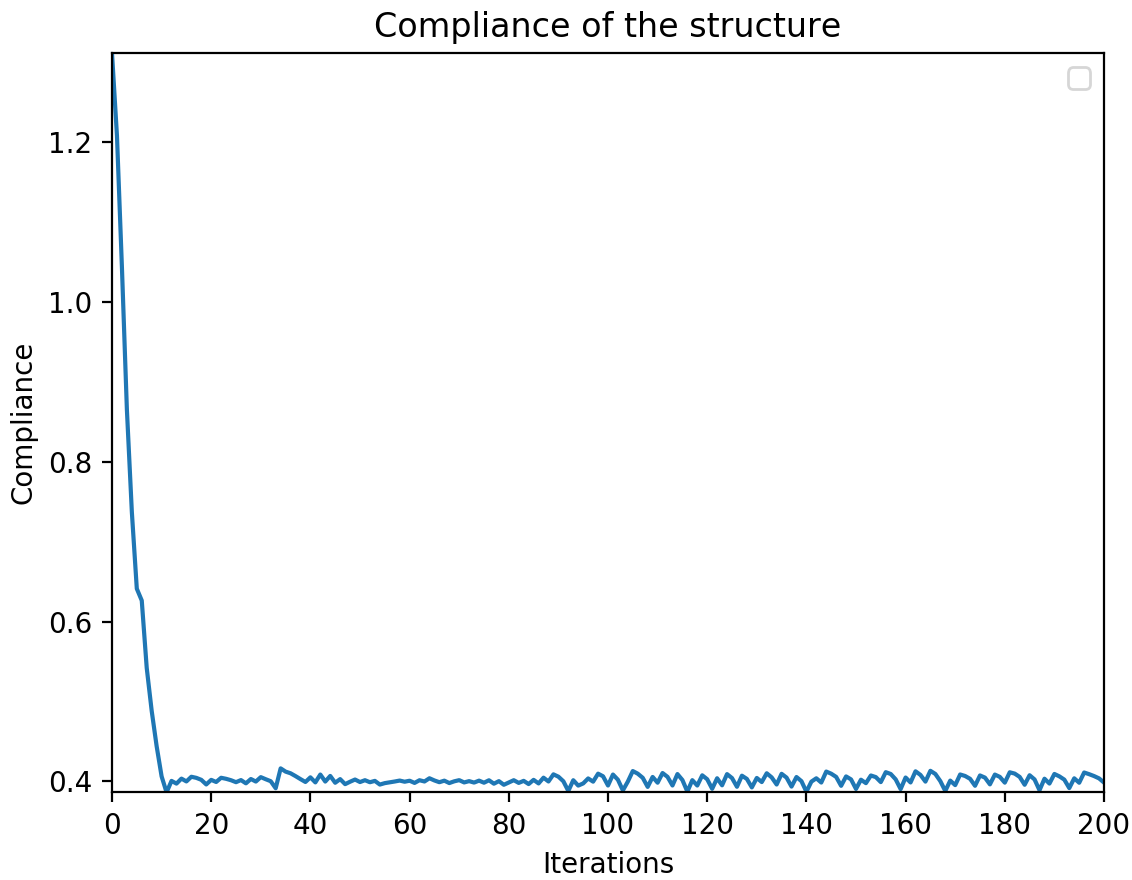}
\end{minipage} & \begin{minipage}{0.3\textwidth}
\includegraphics[width=1.0\textwidth]{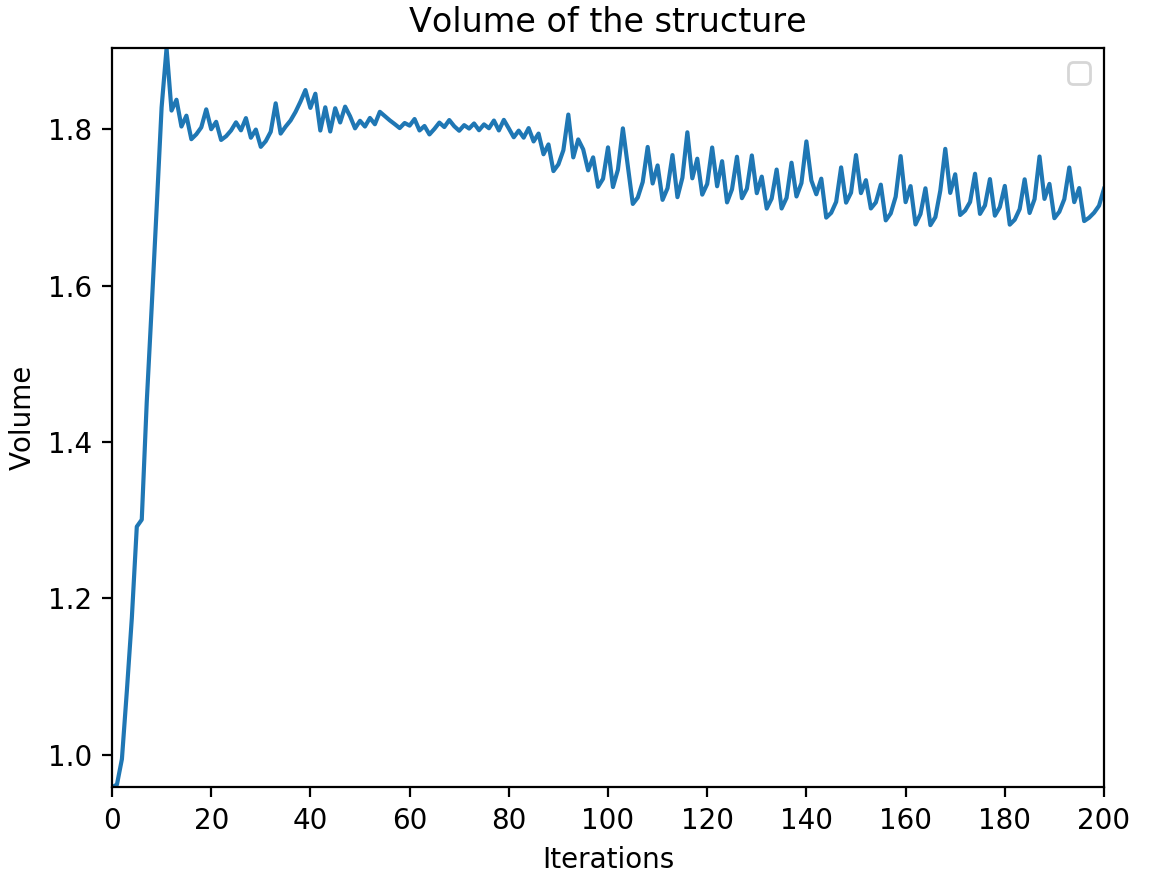}
\end{minipage} 
\end{tabular}
\caption{\it Convergence histories for (left) the compliance in the first stage, (middle) the compliance in the second stage and (right) the volume in the mast example of \cref{sec.mastex}.}
\label{fig.mastexcv}
\end{figure}

\par\medskip
\noindent \textbf{Acknowledgements.} This work was partially supported by the project ANR-18-CE40-0013 SHAPO financed by the French Agence Nationale de la Recherche (ANR).

%
\bibliographystyle{siam}
\bibliography{./genbib}

\end{document}